\documentclass[leqno,10pt]{amsart}

\usepackage{amsfonts}
\usepackage[latin1]{inputenc}
\usepackage{amsmath}
\usepackage{amsfonts}
\usepackage{amssymb}
\usepackage{amscd}
\usepackage{amsthm}
\newcommand{\pa}[1]{{\left(#1\right)}}                  % tra tonde
\newcommand{\abs}[1]{{\left|#1\right|}}                 % valore assoluto
\newcommand{\pair}[1]{\left\langle#1\right\rangle}      % pairing

\newcommand{\metric}{\left\langle\,,\,\right\rangle}                          % metrica
\newcommand{\riemann}[1]{\pa{#1,\metric}}                 % varietà riemanniana
\newtheorem{theorem}{Theorem}
\newtheorem{corollary}[theorem]{Corollary}
\newtheorem{lemma}[theorem]{Lemma}
\newtheorem{proposition}[theorem]{Proposition}
\newtheorem{remark}[theorem]{Remark}

\newcommand{\SSS}{\mathrm S}
\newcommand{\TTT}{\mathrm T}
\newcommand{\RRR}{\mathrm R}
\newcommand{\WWW}{\mathrm W}
\newcommand{\Ric}{\operatorname{Ric}}
\newcommand{\dist}{\operatorname{dist}}

\begin{document}
\title{Some geometric analysis on generic Ricci solitons}

%\author{Paolo Mastrolia \and Marco Rigoli \and Michele Rimoldi}

\author{Paolo Mastrolia}
\address{Dipartimento di Matematica\\ Universit\`a degli Studi di Milano\\ via Saldini 50, I-20133 Milano, Italy }
                          \email[Paolo Mastrolia]{paolo.mastrolia@gmail.com}

\author{Marco Rigoli}
\address{Dipartimento di Matematica\\ Universit\`a degli Studi di Milano\\ via Saldini 50, I-20133 Milano, Italy }
                          \email[Marco Rigoli]{marco.rigoli@unimi.it}

\author{Michele Rimoldi}
\address{Dipartimento di Scienza e Alta Tecnologia\\ Universit\`a degli Studi dell'Insubria\\ via Valleggio 11, I-22100 Como, Italy }
                          \email[Michele Rimoldi]{michele.rimoldi@gmail.com}

%\institute{Paolo Mastrolia \at
%              Dipartimento di Matematica\\ Universit\`a degli Studi di Milano\\ via Saldini 50, I-20133 Milano, Italy \\
%                            \email{paolo.mastrolia@gmail.com}           %  \\
%%             \emph{Present address:} of F. Author  %  if needed
%           \and
%           Marco Rigoli \at
%              Dipartimento di Matematica\\ Universit\`a degli Studi di Milano\\ via Saldini 50, I-20133 Milano, Italy \\
%                            \email{marco.rigoli@unimi.it}
%           \and
%           MIchele Rimoldi \at
%              Dipartimento di Matematicazz Universit\`a degli Studi di Milano\\ via Saldini 50, I-20133 Milano, Italy \\
%                            \email{michele.rimoldi@unimi.it}
%}

\date{\today}

\begin{abstract}
We study the geometry of complete generic Ricci solitons with the aid of some geometric--analytical tools extending techniques of the usual Riemannian setting.
%\subclass{53C21}
\end{abstract}
%\thanks{Corresponding Author: Michele Rimoldi, Dipartimento di Scienza e Alta Tecnologia, Università degli Studi dell'Insubria,
%via Valleggio 11, I-22100 Como, Italy. Email: michele.rimoldi@gmail.com.}

\maketitle

\section{Introduction and main results}\label{Sect1}

Let $(M, \left\langle \,,\,\right\rangle)$ be an $m$-dimensional, complete, connected Riemannian manifold. A \emph{soliton structure} $(M, \left\langle \,,\,\right\rangle, X)$ on $M$ is the choice (if any) of a smooth vector field $X$ on $M$ and a real constant $\lambda$ such that
\begin{equation}\label{RicSolEq}
  \Ric + \frac 12 \,\mathcal{L}_X\left\langle \,,\,\right\rangle = \lambda \left\langle \,,\,\right\rangle,
\end{equation}
where $\Ric$ denotes the Ricci tensor of the metric $\left\langle \,,\,\right\rangle$ on $M$ and $\mathcal{L}_X\left\langle \,,\,\right\rangle$ is the Lie derivative of this latter in the direction of $X$. In what follows we shall refer to $\lambda$ as to the \emph{soliton constant}. The soliton is called \emph{expanding}, \emph{steady} or \emph{shrinking} if, respectively, $\lambda < 0$, $\lambda = 0$ or $\lambda >0$. If $X$ is the gradient of a potential $f \in C^{\infty}\pa{M}$, \eqref{RicSolEq} takes the form
\begin{equation}\label{GradRicSolEq}
  \Ric + \operatorname{Hess}(f) = \lambda \left\langle \,,\,\right\rangle,
\end{equation}
and the Ricci soliton is called a \emph{gradient Ricci soliton}. Both equations \eqref{RicSolEq} and \eqref{GradRicSolEq} can be considered as perturbations of the Einstein equation
\[
\Ric = \lambda\left\langle \,,\,\right\rangle
\]
and reduce to this latter in case $X$ or $\nabla f$ are Killing vector fields. When $X=0$ or $f$ is constant we call the underlying Einstein manifold a \emph{trivial} Ricci soliton.

Since the appearance of the seminal works of R. Hamilton, \cite{hamilton}, and G. Perelman, \cite{perelman}, the study of gradient Ricci solitons has become the subject of a rapidly increasing investigation mainly directed towards two goals, \emph{classification} and \emph{triviality}; among the enormous literature on the subject we only quote, as a few examples, the papers \cite{ELnM}, \cite{petwylie1}, \cite{petwylie2}, \cite{petwylie3}, \cite{PRRS}, \cite{PRiS}, which, in some sense, are the most related to the present work.

On the other hand relatively little is known about generic Ricci solitons, that is, when $X$ is not necessarily the gradient of a potential $f$, and the majority of the results is concerned with the compact case. For instance, it is well known that generic expanding and steady compact Ricci solitons are trivial (see e.g. \cite{petwylie2}). It is also worth pointing out that on a compact manifold shrinking Ricci solitons always support a gradient soliton structure, \cite{perelman1}, and that every complete non-compact shrinking Ricci soliton with bounded curvature supports a gradient soliton structure, \cite{Na}. Observe that expanding Ricci solitons which do not support a gradient soliton structure were found by J. Lauret, \cite{L} and P. Baird and L. Danielo, \cite{BD}. These spaces exhibit left invariant metrics on Sol and Nil manifolds.

A first important difference is that, in the general case, we cannot make use of the weighted manifold structure $(M, \left\langle \,,\,\right\rangle, e^{-f}d\rm{vol})$ which naturally arises when dealing with gradient solitons. The same applies for related concepts such as the Bakry--Emery Ricci tensor, whose boundedness from below with a suitable radial function, together with an additional assumption on the potential function $f$, gives rise to weighted volume estimates (see \cite{Q}, \cite{WW}, \cite{PRRS}). These facts restrict the applicability of analytical tools such as the weak maximum principle for the diffusion operator $\Delta_f$, weighted $L^p$ Liouville--type theorems and \emph{a priori} estimates that have been considered in previous investigations, (see \cite{PRiS}, \cite{PRRS} for details). Nevertheless, in the general case the soliton structure is encoded in the geometry of an appropriate operator $\Delta_X$ that we shall call the ``$X$--Laplacian'' and that is defined on $u\in C^2(M)$ by
\[
\ \Delta_X u=\Delta u -\pair{X, \nabla u}.
\]
Clearly, if $X$ is the gradient of some function $f$, $\Delta_X$ reduces to the $f$--Laplacian.

In the present work, assuming a suitable growth condition on the vector field $X$,
%by means of a result that gives function-theoretic sufficient conditions for the validity of the Omori--Yau maximum principle for $\Delta_X$, a Motomiya--type result for solutions of $\Delta_X$--Poisson equations and a comparison result for the $X$--laplacian of the distance function from a fixed reference point,
we prove two results (see Theorem \ref{ThA} and \ref{ThB} below) on general solitons which are sharp enough to recover the corresponding results in \cite{PRiS}, \cite{PRRS} for gradient solitons. Towards this goal we introduce some analytical tools: a function theoretic version of the Omori-Yau maximum principle for $\Delta_X$ (Lemma \ref{FullOmoriYau}), a comparison result for the same operator (Lemma \ref{2.18}), and an \emph{a priori} estimate (Lemma \ref{Motomiya}) similar to that of Theorem 1.31 in \cite{prsMemoirs}. These results should be useful also in other settings (see Proposition \ref{2.52} below and its consequences).

From now on we fix an origin $o \in M$ and let $r(x) = \dist(x, o)$. We set $B_r$ and $\partial B_r$ to denote, respectively, the geodesic ball of radius $r$ centered at $o$ and its boundary.

We are now ready to state our

\begin{theorem}\label{ThA}
  Let $(M, \left\langle \,,\,\right\rangle)$ be a complete Riemannian manifold of dimension $m$ with scalar curvature $\SSS(x)$ and $(M, \left\langle \,,\,\right\rangle, X)$ a soliton structure on $M$ with soliton constant $\lambda$. Assume
  \begin{equation}\label{BoundX}
    |X| \leq \sqrt{G(r)}
  \end{equation}
where $G$ is a smooth function on $\left[0,+\infty\right)$ satisfying
\begin{equation}\label{hpG}
\begin{array}{lll}
&\left(i\right)\,G\left(0\right)>0&\left(ii\right)\,G^{\prime}\left(t\right)\geq 0 \textrm{\,\,on\,\,} \left[0,+\infty\right)\\
&\left(iii\right)G\left(t\right)^{-\frac{1}{2}}\notin L^{1}\left(+\infty\right)&\left(iv\right)\, \limsup_{t\rightarrow+\infty}\frac{tG\left(t^{\frac{1}{2}}\right)}{G\left(t\right)}<+\infty.
\end{array}
\end{equation}
Let
\[
\ \SSS_* = \inf_M \SSS.
\]
  \begin{itemize}
    \item [(i)] \, If $\lambda <0$ then $m\lambda \leq \SSS_* \leq 0$. Furthermore, if $\SSS(x_0) = \SSS_*=m\lambda$ for some $x_0 \in M$, then $(M, \left\langle \,,\,\right\rangle)$ is Einstein and $X$ is a Killing field; while if $\SSS(x_0)=\SSS_{*}=0$ for some $x_0\in M$, then $(M, \left\langle \,,\,\right\rangle)$ is Ricci flat and $X$ is a homothetic vector field.
    \item [(ii)] If $\lambda =0$ then $\SSS_*= 0$. Furthermore, if $\SSS(x_0) = \SSS_*=0$ for some $x_0 \in M$, then $(M, \left\langle \,,\,\right\rangle)$ is Ricci flat and $X$ is a Killing field.
    \item [(iii)] If $\lambda >0$ then $0 \leq \SSS_* \leq m\lambda$. Furthermore, if $\SSS(x_0) = \SSS_*=0$ for some $x_0 \in M$, then $(M, \left\langle \,,\,\right\rangle)$ is Ricci flat and $X$ is a homothetic vector field, while  $\SSS_*<m\lambda$ unless $(M, \left\langle \,,\,\right\rangle)$ is compact, Einstein and $X$ is a Killing field.
  \end{itemize}
\end{theorem}

\begin{remark}\label{RmkSharp}
\rm{
Defining
\[
\ G(t)=At^2\prod_{j=1}^{N}(\log^{j}(t))^2,
\]
for some constant $A>0$ and $t\gg1$, where $\log^{(j)}$ stands for the $j$--th iterated logarithm, and completing appropriately the definition on $[0, +\infty)$, we obtain a family of functions satisfying \eqref{hpG}. In particular this is true for
\[
\ G(t)=1+t^2 \;\textrm{on}\; [0, +\infty).
\]
Note that by Z.--H. Zhang, \cite{ZHZhang}, in case the soliton is a gradient soliton, that is, $X=\nabla f$ for some potential $f\in C^{\infty}(M)$, then
\begin{equation}\label{Zhang}
|\nabla f|\leq c(1+ r(x)),
\end{equation}
for some constant $c>0$. Hence for a gradient soliton the upper
bound \eqref{BoundX} is automatically satisfied. In this way we
recover the scalar curvature estimates of Theorem 3 in \cite{PRiS}
and Theorem 1.4 in \cite{PRRS}.}
\end{remark}

\begin{remark}\rm{ In general for complete Ricci solitons which are not necessarily gradient Ricci solitons, we have no control on the growth rate of the norm of the soliton field $X$. Consider now the vector field $Y\in \mathfrak{X}(M\times J_1)$, given by $Y(x,t)=\frac{X(x)}{1-2\lambda t}+\frac{\partial}{\partial t}$, where $J_1\subseteq\mathbb{R}$ is defined by
\begin{equation*}
J_1=\left\{\begin{array}{ll}
\left(\frac{1}{2\lambda},+\infty\right)&\textrm{if }\lambda<0\\
\mathbb{R}&\textrm{if }\lambda=0\\
\left(-\infty,\frac{1}{2\lambda}\right)&\textrm{if }\lambda>0.
\end{array}\right.
\end{equation*}
Our claim is now that the requirement
 \begin{equation}\label{ZhangX}
 |X|\leq c(1+r(x)),
 \end{equation}
 for some constant $c>0$, permits to conclude, as in \cite{ZHZhang}, that for every $t_0<\frac{1}{2\lambda}$, $t_1>\frac{1}{2\lambda}$ fixed, and for every $t\in\mathbb{R}$ there exist diffeomorphisms $\psi_t:M\times J_2\to M\times J_2$ such that $\psi_0=id_{M\times J_2}$ and $\frac{d}{dt}\psi_t=Y\circ\psi_t$ on $M\times J_2$, where
 \begin{equation*}
J_2=\left\{\begin{array}{ll}
\left(t_1,+\infty\right)&\textrm{if }\lambda<0\\
\mathbb{R}&\textrm{if }\lambda=0\\
\left(-\infty,t_0\right)&\textrm{if }\lambda>0.
\end{array}\right.
\end{equation*}
Toward this aim we have to show that for any fixed $(y,\overline{t})\in M\times J_2$ the maximal interval (containing $0$) $J((y,\overline{t}))=(a((y,\overline{t})),b((y,\overline{t})))$ where the integral curve of $Y$ emanating from $(y,\overline{t})$ is defined, coincides with $J_1$. Let us suppose by contradiction e.g. that, in case $\lambda\leq0$, $b((y,\overline{t}))<+\infty$. By a well--known ``escape'' lemma (see e.g. Lemma 12.11 in \cite{Lee}) we then know that the integral curve $\Phi_{(y,\overline{t})}:J((y,\overline{t}))\to M\times\mathbb{R}$ is a divergent curve. Now, let $\varepsilon=\inf\left\{s\in J((y,\overline{t})):\Phi_{(y,\overline{t})}(s)\in (\,^MB_1(y))^{c}\times J_2\right\}$ and for every $t<b((y,\overline{t}))$ consider the restriction
 \[
 \ \gamma=\left.\Phi_{(y,\overline{t})}\right|_{[\varepsilon, t]}:[\varepsilon, t]\to M\times\mathbb{R}.
 \]
Then
\[
\ l(\gamma)=\int_{\varepsilon}^{t}|\dot{\gamma}(s)|ds.
\]
By \eqref{ZhangX} we have that outside $\,^MB_1(y)$
\begin{equation*}
|X(x)|\leq 2c r(x).
\end{equation*}
Let $\tilde{r}((x,t))=d_{M\times\mathbb{R}}\left((x,t),(y,\overline{t})\right)$. Hence we have that in $(\,^MB_{1}(y))^c\times J_2$
\begin{equation*}
|Y((x,t))|\leq B\tilde{r}\left((x,t) \right),
\end{equation*}
for some constant $B>0$.
Thus
\begin{eqnarray*}
\tilde{r}(\gamma(t))&\leq& d_{M\times\mathbb{R}}(\gamma(t), (y,\overline{t}))\\
&\leq& \tilde{r}\left(\gamma(\varepsilon)\right)+d_{M\times\mathbb{R}}\left(\gamma(\varepsilon),\gamma(t)\right)\\
&\leq&\tilde{r}\left(\gamma(\varepsilon)\right)+l(\gamma)\\
&=&\tilde{r}\left(\gamma(\varepsilon)\right)+\int_{\varepsilon}^{t}|\dot{\gamma}(s)|ds\\
&=&\tilde{r}\left(\gamma(\varepsilon)\right)+\int_{\varepsilon}^{t}|Y(\gamma(s))|ds\\
&\leq& \tilde{r}\left(\gamma(\varepsilon)\right)+B\int_{\varepsilon}^t\tilde{r}(\gamma(s))ds.
\end{eqnarray*}
Writing this in terms of $I(t)=\int_{\varepsilon}^t\tilde{r}(\gamma(s))ds$ and integrating the resulting differential inequality one obtains
\[
\ \int_{\varepsilon}^{t}\tilde{r}\left(\gamma(s)\right)ds\leq B^{\prime}e^{Bt},
\]
for some constant $B^{\prime}>0$. Recalling that $\gamma$ has to be divergent we thus get that $b((y,\overline{t}))=+\infty$, getting the desired contradiction.

By the definition of $Y$ we have that for every $(x,t)\in M\times J_2$ the diffeomorphisms $\psi_t$, $t\in J_2$, can be written in the form
\[
\ \psi_t(x, t)=\left(\varphi_t(x), t\right),
\]
for some diffeomorphisms $\varphi_t:M\to M$. Moreover we have that $\varphi_0=id_M$ and $\frac{d}{dt}\varphi_t(x)=\frac{1}{1-2\lambda t}X(x)$.
Let now $\metric(t)$ be defined by
\begin{equation*}
\metric(t)=(1-2\lambda t)\varphi_t^{*}\metric,
\end{equation*}
We then have that
\begin{eqnarray*}
\frac{d}{dt}\metric(t)&=&\frac{d}{dt}\left((1-2\lambda t)\varphi_t^*\metric\right)\\&=&-2\lambda\varphi_t^*\metric+(1-2\lambda t)\varphi_t^*(\mathcal{L}_{\frac{X}{1-2\lambda t}}\metric)\\
&=&-2\varphi_t^*(\lambda \metric -\frac{(1-2\lambda t)}{2}\mathcal{L}_{\frac{X}{1-2\lambda t}}\metric)\\
&=&-2\varphi_t^*(\lambda \metric-\frac{1}{2}\mathcal{L}_X \metric)\\
&=&-2\varphi_t^*(\Ric(\metric))=-2(\Ric(\metric(t))).
\end{eqnarray*}

Thus for every $t_0<\frac{1}{2\lambda}$ and $t_1>\frac{1}{2\lambda}$ we can define a self--similar solution of the Ricci flow $(M, \metric(t))$ defined respectively on $(-\infty, t_0)$ if $\lambda>0$, on $\mathbb{R}$ if $\lambda=0$ and on $(t_1, +\infty)$ if $\lambda<0$. In particular, a complete Ricci soliton $(M, \left\langle \,,\,\right\rangle, X)$ for which \eqref{ZhangX} holds always corresponds to the ``self--similar'' solution of the Ricci flow it generates.
}
\end{remark}

Our next result is a gap theorem for the traceless Ricci tensor.
\begin{theorem}\label{ThB}
Let $(M, \left\langle \,,\,\right\rangle)$ be a complete Riemannian manifold of dimension $m\geq 3$, scalar curvature $S(x)$ and trace free Ricci tensor $T$. Suppose that
\begin{equation}\label{supS}
i)\, S^{*}=\sup_M S(x)<+\infty; \qquad ii)\, \abs{\WWW}^* = \sup_M \abs{\WWW} <+\infty,
\end{equation}
where $W$ is the Weyl tensor of $\riemann{M}$.
Let $(M, \metric, X)$ be a soliton structure on $M$ with soliton constant $\lambda$. Assume
\begin{equation}\label{BoundX1}
|X|\leq \sqrt{G(r)}
\end{equation}
where $G$ satisfies $\eqref{hpG}$. Then, either $(M, \metric)$ is Einstein or $|T|^*=\sup_M |T|$ satisfies
\begin{equation}\label{1.10}
|T|^*\geq\frac{1}{2}\pa{\sqrt{m(m-1)}\lambda-\frac{m-2}{\sqrt{m(m-1)}}\,\SSS^{*}-\sqrt{\frac{m(m-2)}{2}}\abs{\WWW}^*}.
\end{equation}
\end{theorem}
In case $\riemann{M}$ is conformally flat, since Remark
\ref{RmkSharp} above applies again this result recovers Theorem
1.9 of \cite{PRRS} for gradient Ricci solitons.

As an immediate consequence of Theorem \ref{ThA} we have
\begin{corollary}
Let $(M, \left\langle \,,\,\right\rangle)$ be a complete Riemannian manifold admitting a shrinking or steady soliton structure $(M, \left\langle \,,\,\right\rangle, X)$ satisfying \eqref{BoundX1} and \eqref{hpG}. Then any minimal immersion of $(M, \left\langle \,,\,\right\rangle)$ into $\mathbb{R}^n$, $n>m=dim M$ is totally geodesic.
\end{corollary}
\begin{proof}
Indicating by $II$ the second fundamental tensor of the immersion by Gauss equations we have
\[
\ \SSS(x)=-|II|^2(x).
\]
Thus if the immersion is not totally geodesic $\SSS_{*}<0$, contradicting (iii) or (ii) of Theorem \ref{ThA}.
\end{proof}
Because of \eqref{Zhang} the above Corollary specializes to
\begin{corollary}
Let $(M, \left\langle \,,\,\right\rangle)$ be a complete Riemannian manifold admitting a shrinking or steady gradient soliton structure $(M, \left\langle \,,\,\right\rangle, \nabla f)$. Then any minimal immersion of $(M, \left\langle \,,\,\right\rangle)$ into $\mathbb{R}^n$, $n>m=dim M$ is totally geodesic.
\end{corollary}

\section{Some basic formulas}
The aim of this section is to collect and prove some basic formulas for generic Ricci solitons. Formula \eqref{eq8} of Lemma \ref{XlapRicS} and formula \eqref{XlapT1} of Lemma \ref{XlapT} are the basic ingredients in the proofs respectively of Theorems \ref{ThA} and \ref{ThB}. Their derivation exploit the symmetries of the curvature tensor, in particular the second Bianchi identity, coupled with the soliton equation \eqref{RicSolEq} via covariant differentiation. Since the process is quite involved we have divided the proofs in a number of steps. In Lemma \ref{LemmaComm} below we recall some standard commutation relations while in Lemma \ref{LemmaId} we prove new commutation rules which are related to the soliton structure. These allow us, after some efforts, to prove the basic equations \eqref{eq7} and \eqref{eq8} of Lemma \ref{XlapRicS} and \eqref{XlapT1} of Lemma \ref{XlapT}. These formulas seem to be new and of independent interest. Furthermore, they also make clear the naturality, in this context, of the operator $\Delta_X$ defined in Section \ref{Sect1}. 

In what follows, to perform computations, we shall use the method of the moving frame referring to a local orthonormal coframe. Here we fix the index range $1\leq i, j, \ldots \leq m = \operatorname{dim}M$ and we use the Einstein summation convention throughout.

\begin{lemma}\label{LemmaComm} Let $X$ be a vector field on a Riemannian manifold $(M, \left\langle \,,\,\right\rangle)$. Then the following commutation formulas hold:
\begin{eqnarray}
&\label{a:1} X_{ijk}-X_{ikj}=X_{t}\RRR_{tijk};\\
&\label{a:2} X_{ijkl}-X_{ikjl}=\RRR_{tijk}X_{tl}+\RRR_{tijk,l}X_{t}; \\
&\label{a:3} X_{ijkl}-X_{ijlk}=\RRR_{tikl}X_{tj}+\RRR_{tjkl}X_{it}.
\end{eqnarray}
\end{lemma}
\begin{proof}
Equation \eqref{a:2} follows by taking the covariant derivative of \eqref{a:1}.

Equations \eqref{a:1} and \eqref{a:3} are the standard commutation formulas of covariant derivatives acting on tensors.
\end{proof}

Moreover we have the validity of the following general Ricci identities:
\begin{eqnarray}
&\label{a:4} \RRR_{ij,k}=\RRR_{ji,k};\\
&\label{a:5} \RRR_{ij,k}-\RRR_{ik,j}=-\RRR_{tijk,t}; \\
&\label{a:6} \RRR_{ij,kl}-\RRR_{ij,lk}=\RRR_{tikl}\RRR_{tj}+\RRR_{tjkl}\RRR_{it}; \\
&\label{a:7} \frac 12 \SSS_k = \RRR_{ki, i} = \RRR_{ik, i}.
\end{eqnarray}

Equation \eqref{a:4} is obvious from the symmetry of the Ricci tensor i.e. $\RRR_{ij}=\RRR_{ji}$.

Equation \eqref{a:5} follows from the second Bianchi's identities, indeed
\begin{eqnarray*}
\RRR_{ij,k}-\RRR_{ik,j}&=&\RRR_{titj,k}-\RRR_{titk,j}\\
&=&\RRR_{titj,k}+\RRR_{tikt,j}\\
&=&-\RRR_{tijk,t}.
\end{eqnarray*}
Equation \eqref{a:6} can be obtained using the definition of covariant derivative and the structure equations.

\noindent Equation \eqref{a:7} follows tracing \eqref{a:5}.

We are now ready to prove

\begin{lemma}\label{LemmaId} Let $(M,\left\langle \,,\,\right\rangle, X)$ be a soliton structure on the Riemannian manifold $(M, \metric)$. Then the following identities hold:
\begin{eqnarray}
&\label{eq1}\RRR_{ij} +\frac{1}{2}(X_{ij}+X_{ji})=\lambda \delta_{ij},\\
&\label{eq2}\SSS +  X_{ii}=m\lambda, \\
&\label{eq3}\SSS_j = -X_{iij}, \\
&\label{eq4}\RRR_{lj}X_{l}=-X_{jii},\\
&\label{eq5}\RRR_{ij,k}-\RRR_{ik,j}=-\frac{1}{2}\RRR_{lijk}X_{l}+\frac{1}{2}(X_{kij}-X_{jik}), \\
&\label{eq6}\RRR_{ij,k}-\RRR_{kj,i}=\frac{1}{2}\RRR_{ljki}X_{l}+\frac{1}{2}(X_{kji}-X_{ijk}).
\end{eqnarray}
\end{lemma}

\begin{proof}
Equation \eqref{eq1} is a rewriting of \eqref{RicSolEq}. Equation \eqref{eq2} follows simply by tracing \eqref{eq1}.\\
To obtain \eqref{eq3} take the covariant derivative of \eqref{eq2}.\\
Next, taking the trace of the commutation formula \eqref{a:1} with respect to $i$ and $k$ we get
\[
\ X_{iji}-X_{iij}=\RRR_{liji}X_{l}=\RRR_{lj}X_{l}.
\]
Using this latter and equation \eqref{eq1} we compute
\begin{eqnarray*}
\RRR_{ij,i}&=&-\frac{1}{2}(X_{iji}+X_{jii})\\
&=&-\frac{1}{2}(X_{iji}-X_{iij}+X_{iij}+X_{jii})\\
&=&-\frac{1}{2}\RRR_{lj}X_{l}-\frac{1}{2}(X_{iij}+X_{jii}).
\end{eqnarray*}
It follows by \eqref{a:7} that
\[
\ \SSS_j=-\RRR_{lj}X_{l}-(X_{iij}+X_{jii})
\]
and using \eqref{eq3} we deduce \eqref{eq4}.\\
To prove \eqref{eq5} observe that taking the covariant derivative of \eqref{eq1} yields
\[
\RRR_{ij,k}+\frac{1}{2}(X_{ijk}+X_{jik})=0
\]
and
\[
\RRR_{ik,j}+\frac{1}{2}(X_{ikj}+X_{kij})=0.
\]
Taking the difference and using \eqref{eq1} we have
\begin{eqnarray*}
\RRR_{ij,k}-\RRR_{ik,j}&=&-\frac{1}{2}(X_{ijk}+X_{jik}-X_{ikj}-X_{kij})\\
&=&-\frac{1}{2}(X_{ijk}-X_{ikj})+\frac{1}{2}(X_{kij}-X_{jik})\\
&=&-\frac{1}{2}\RRR_{lijk}X_{l}+\frac{1}{2}(X_{kij}-X_{jik}),
\end{eqnarray*}
that is, \eqref{eq5}.

Similarly, using the commutation formula \eqref{a:1} we obtain
\begin{eqnarray*}
2(\RRR_{ij,k}-\RRR_{kj,i})&=&-(X_{ijk}+X_{jik}-X_{kji}-X_{jki})\\
&&=X_{kji}-X_{ijk}+X_{t}\RRR_{tjki},
\end{eqnarray*}
that is, \eqref{eq6}.
\end{proof}

\begin{lemma}\label{XlapRicS}
Let $(M,\left\langle \,,\,\right\rangle, X)$ be a soliton structure on the Riemannian manifold $(M, \metric)$,  then
\begin{eqnarray}
&\label{eq7}\Delta_X\RRR_{ik}=2\lambda\RRR_{ik}-2\RRR_{ijks}\RRR_{js}+\frac{1}{2}\RRR_{is}(X_{sk}-X_{ks})+\frac{1}{2}\RRR_{sk}(X_{si}-X_{is}),\\
&\label{eq8}\Delta_X\SSS=2\lambda \SSS-2|Ric|^2.
\end{eqnarray}
\end{lemma}

\begin{proof}
To prove Equation \eqref{eq7} first observe that by \eqref{eq5}
\[
\ \RRR_{ki,j}-\RRR_{kj,i}=\frac{1}{2}\RRR_{lkji}X_{l}+\frac{1}{2}(X_{jki}-X_{ikj}),
\]
and thus taking the covariant derivative we obtain the commutation relations
\begin{equation}\label{comm2ric} \RRR_{ik,jt}-\RRR_{jk,it}=\frac{1}{2}(\RRR_{ijkl,t}X_{l}+\RRR_{ijkl}X_{lt})+\frac{1}{2}(X_{jkit}-X_{ikjt}).
\end{equation}
Moreover, contracting the commutation relations \eqref{a:6} for the second covariant derivative of $R_{jk}$ \begin{eqnarray}\label{25'}
\RRR_{jk,ij}&=&\RRR_{jk,ji}+\RRR_{sjij}\RRR_{sk}+\RRR_{skij}\RRR_{js}\\
\nonumber &=&\RRR_{jk,ji}+\RRR_{si}\RRR_{sk}+\RRR_{skij}\RRR_{js}.
\end{eqnarray}
We now use \eqref{comm2ric} to obtain
\begin{equation*}
\Delta\RRR_{ik}=\RRR_{ik,jj}=\RRR_{jk,ij}+\frac{1}{2}(\RRR_{ijkl,j}X_{l}+\RRR_{ijkl}X_{lj})+\frac{1}{2}(X_{jkij}-X_{ikjj}).
\end{equation*}
On the other hand, from the second Bianchi identities
\begin{eqnarray*}
\RRR_{ijkl,j}X_{l}&=&-\RRR_{ijlj,k}X_{l}-\RRR_{ijjk,l}X_{l}\\
&=&-\RRR_{il,k}X_{l}+\RRR_{ik,l}X_{l},
\end{eqnarray*}
and inserting into the above, with the aid of \eqref{25'}, yields
\begin{eqnarray*}
\Delta\RRR_{ik}&=&\RRR_{jk,ij}+\frac{1}{2}(\RRR_{ik,l}-\RRR_{il,k})X_{l}+\frac{1}{2}\RRR_{ijkl}X_{lj}+\frac{1}{2}(X_{jkij}-X_{ikjj})\\
&=&\RRR_{jk,ji}+\RRR_{si}\RRR_{sk}+\RRR_{skij}\RRR_{js}+\frac{1}{2}(\RRR_{ik,l}-\RRR_{il,k})X_{l}\\
&&+\frac{1}{2}\RRR_{ijkl}X_{lj}+\frac{1}{2}(X_{jkij}-X_{ikjj}).
\end{eqnarray*}
Hence, from \eqref{a:7} and \eqref{eq1}
\begin{eqnarray*}
\Delta\RRR_{ik}&=&\frac{1}{2}\SSS_{ki}+\RRR_{si}\RRR_{sk}+\RRR_{skij}\RRR_{sj}-\frac{1}{2}\RRR_{skij}X_{sj}\\
&&+\frac{1}{2}(\RRR_{ik,l}-\RRR_{il,k})X_{l}+\frac{1}{2}(X_{jkij}-X_{ikjj})\\
&=&\frac{1}{2}\SSS_{ki}+\RRR_{si}\RRR_{sk}+2\RRR_{skij}\RRR_{sj}-\RRR_{skij}(-\RRR_{sj}+\lambda \delta_{sj}-\frac{1}{2}X_{js})\\
&&+\frac{1}{2}(\RRR_{ik,l}-\RRR_{il,k})X_{l}+\frac{1}{2}(X_{jkij}-X_{ikjj})\\
&=&\frac{1}{2}\SSS_{ki}+\RRR_{si}\RRR_{sk}+2\RRR_{skij}\RRR_{sj}+\lambda\RRR_{ik}+\frac{1}{2}\RRR_{skij}X_{js}+\frac{1}{2}(\RRR_{ik,s}-\RRR_{is,k})X_{s}\\
&&+\frac{1}{2}(X_{jkij}-X_{ikjj}).
\end{eqnarray*}
We shall now deal with the sum
\begin{equation}\label{Z}
Z=\frac{1}{2}\SSS_{ik}+\RRR_{sk}\RRR_{si}-\frac{1}{2}\RRR_{is,k}X_{s}+\frac{1}{2}X_{skis}.
\end{equation}
Towards this aim we first observe that taking the covariant derivative of \eqref{eq3} gives
\begin{equation*}
\frac{1}{2}\SSS_{ik}=-\frac{1}{2}X_{ssik}.
\end{equation*}
Using this information in \eqref{Z} we obtain
\begin{eqnarray*}
Z&=&-\frac{1}{2}X_{ssik}+\RRR_{sk}\RRR_{si}-\frac{1}{2}\RRR_{is,k}X_{s}+\frac{1}{2}X_{skis}\\
&=&\frac{1}{2}(X_{skis}-X_{ssik})+\RRR_{sk}\RRR_{si}-\frac{1}{2}\RRR_{is,k}X_{s} \\
&=& \frac{1}{2}(X_{skis}-X_{sksi}+X_{sksi}-X_{sski})+\RRR_{sk}\RRR_{si}-\frac{1}{2}\RRR_{is,k}X_{s}\\
&=& \frac{1}{2}(\RRR_{tsis}X_{tk}+\RRR_{tkis}X_{st}+\RRR_{tsks}X_{ti}+\RRR_{tsks,i}X_{t})+\RRR_{sk}\RRR_{si}-\frac{1}{2}\RRR_{is,k}X_{s}\\
&=& \frac{1}{2}(\RRR_{ti}X_{tk}+\RRR_{tkis}X_{st}+\RRR_{tk}X_{ti}+\RRR_{tk,i}X_{t})+\RRR_{sk}\RRR_{si}-\frac{1}{2}\RRR_{is,k}X_{s}\\
&=& \frac{1}{2}(\RRR_{ti}X_{tk}+\RRR_{tkis}X_{st}+\RRR_{tk}X_{ti})+\frac{1}{2}(\RRR_{sk,i}-\RRR_{si,k})X_{s}+\RRR_{sk}\RRR_{si}\\
&=& \frac{1}{2}(\RRR_{si}X_{sk}+\RRR_{tkis}X_{st}+\RRR_{sk}X_{si})-\frac{1}{2}\RRR_{tski,t}X_{s}+\RRR_{sk}\RRR_{si}\\
&=& \RRR_{si}(\RRR_{sk}+\frac{1}{2}X_{sk})+\frac{1}{2}(\RRR_{tkis}X_{st}+\RRR_{sk}X_{si})-\frac{1}{2}\RRR_{tski,t}X_{s}\\
&=& \RRR_{si}(-\frac{1}{2}X_{sk}+\lambda \delta_{sk})+\frac{1}{2}(\RRR_{tkis}X_{st}+\RRR_{sk}X_{si})-\frac{1}{2}\RRR_{tski,t}X_{s}\\
&=& -\frac{1}{2}\RRR_{si}X_{ks}+\lambda\RRR_{ik}+\frac{1}{2}\RRR_{tkis}X_{st}+\frac{1}{2}\RRR_{sk}X_{si}-\frac{1}{2}\RRR_{tski,t}X_{s}.
\end{eqnarray*}
Substituting the above into the expression for $\Delta \RRR_{ik}$ we deduce
\begin{eqnarray*}
\Delta\RRR_{ik}&=&-\frac{1}{2}\RRR_{si}X_{ks}+2\lambda\RRR_{ik}+\frac{1}{2}\RRR_{tkis}X_{st}+\frac{1}{2}\RRR_{sk}X_{si}-\frac{1}{2}\RRR_{tski,t}X_{s}\\&&+2\RRR_{skij}\RRR_{sj}+\frac{1}{2}\RRR_{skij}X_{js}+\frac{1}{2}\RRR_{ik,s}X_{s}-\frac{1}{2}X_{ikss}.
\end{eqnarray*}
Now we note that by \eqref{a:2} and \eqref{a:3},
\begin{eqnarray*}
X_{ikss}-X_{issk}&=&X_{ikss}-X_{isks}+X_{isks}-X_{issk}\\
&=&\RRR_{tiks}X_{ts}+\RRR_{tiks,s}X_{t}+\RRR_{tiks}X_{ts}+\RRR_{tsks}X_{it},
\end{eqnarray*}
that is,
\begin{equation*}
X_{ikss}=X_{issk}+\RRR_{tiks}X_{ts}+\RRR_{tiks,s}X_{t}+\RRR_{tiks}X_{ts}+\RRR_{tk}X_{it}.
\end{equation*}
Moreover we know, taking the covariant derivative of \eqref{eq4}, that
\begin{equation*}
X_{issk}=-\RRR_{ti,k}X_{t}-\RRR_{ti}X_{tk}.
\end{equation*}
Thus,
\begin{equation*}
X_{ikss}=-\RRR_{ti,k}X_{t}-\RRR_{ti}X_{tk}+\RRR_{tiks}X_{ts}+\RRR_{tiks,s}X_{t}+\RRR_{tiks}X_{ts}+\RRR_{tk}X_{it},
\end{equation*}
and substituting above, using the first Bianchi identities and \eqref{a:5},
\begin{eqnarray*}
\Delta\RRR_{ik}&=& 2\lambda\RRR_{ik}-2\RRR_{ijks}\RRR_{js}-\frac{1}{2}\RRR_{si}X_{ks}+\frac{1}{2}\RRR_{tkis}X_{st}\\
&&+\frac{1}{2}\RRR_{sk}X_{si}-\frac{1}{2}\RRR_{tski,t}X_{s}+\frac{1}{2}\RRR_{skij}X_{js}+\frac{1}{2}\RRR_{ik,s}X_{s}\\
&&-\frac{1}{2}(-\RRR_{it,k}X_{t}-\RRR_{it}X_{tk}+\RRR_{tiks}X_{ts}+\RRR_{tiks,s}X_{t}+\RRR_{tiks}X_{ts}+\RRR_{tk}X_{it})\\
&=& 2\lambda \RRR_{ik}-2\RRR_{ijks}\RRR_{js}-\frac{1}{2}\RRR_{is}X_{ks}+\frac{1}{2}\RRR_{tkis}X_{st}+\frac{1}{2}\RRR_{sk}X_{si}-\frac{1}{2}\RRR_{tski,t}X_{s} \\
&&+ \frac{1}{2}\RRR_{skij}X_{js}+\frac{1}{2}\RRR_{ik,s}X_{s}+\frac{1}{2}\RRR_{is,k}X_{s}+\frac{1}{2}\RRR_{is}X_{sk}\\
&&-\frac{1}{2}\RRR_{tiks}X_{ts}-\frac{1}{2}\RRR_{tiks,s}X_{t}-\frac{1}{2}\RRR_{tiks}X_{ts}-\frac{1}{2}\RRR_{sk}X_{is}\\
&=&2\lambda\RRR_{ik}-2\RRR_{ijks}\RRR_{js}+\RRR_{ik,s}X_{s}+\frac{1}{2}\RRR_{is}(X_{sk}-X_{ks})+\frac{1}{2}\RRR_{sk}(X_{si}-X_{is})\\
&&+\RRR_{tkis}X_{st}-\RRR_{tiks}X_{ts}+\frac{1}{2}\RRR_{tisk,t}-\frac{1}{2}\RRR_{tisk,t}X_{s},
\end{eqnarray*}
implying \eqref{eq7}. We conclude observing that \eqref{eq8} is easily obtained by tracing \eqref{eq7}.
\end{proof}

Our next step is to compute the $X$-Laplacian of the square norm of the traceless Ricci tensor $T=\Ric -\frac{\SSS}{m}$ using Lemma \ref{XlapRicS}.
 Before stating the lemma we recall the splitting of the Riemann curvature tensor:

\begin{align}\label{27}
\RRR_{ijks}  &  =\operatorname{W}_{ijks} +\frac{1}{m-2}\pa{\RRR_{ik}\delta_{j s}-\RRR_{is}
\delta_{jk}+\RRR_{js}\delta_{ik}-\RRR_{jk}\delta_{is}} \\ \nonumber
& -   \frac{S}{(m-1)(m-2)}\pa{\delta_{ik}\delta_{js}-\delta_{is}\delta_{jk}},
\end{align}
where $\WWW_{ijks}$ are the components of the Weyl curvature tensor.
\begin{lemma}\label{XlapT}
Let $(M, \left\langle \,,\,\right\rangle, X)$ be a Ricci soliton structure on the Riemannian manifold $(M, \metric)$. Then
\begin{equation}\label{XlapT1}
\frac{1}{2}\Delta_X|T|^2=\abs{\nabla T}^2+2\pa{\lambda-\frac{m-2}{m(m-1)}\SSS}\abs{\TTT}^2+\frac{4}{m-2}Tr(\TTT^3)-2\TTT_{ik}\TTT_{sj}\WWW_{ksij}.
\end{equation}
\end{lemma}

\begin{proof}
From the definition of $\TTT$ we have
\begin{eqnarray}
\label{T1}|\TTT|^2&=&|\Ric|^2-\frac{S^2}{m},\\
\label{T2}|\nabla \TTT|^2&=&|\nabla\Ric|^2-\frac{|\nabla S|^2}{m}.
\end{eqnarray}
Using \eqref{eq7} we deduce
\begin{eqnarray}
\label{XlapnormRic}\Delta_X|\Ric|^2&=&2(\Delta\RRR_{ik})\RRR_{ik}+2|\nabla\Ric|^2-2\RRR_{ik}\RRR_{ik,s}X_s\\
&=&2\RRR_{ik}\Delta_X\RRR_{ik}+2|\nabla\Ric|^2\nonumber\\
&=&4\lambda|\Ric|^2-4\RRR_{ijks}\RRR_{ik}\RRR_{js}+\RRR_{is}\RRR_{ik}(X_{sk}-X_{ks})\nonumber\\&&+\RRR_{sk}\RRR_{ik}(X_{si}-X_{is})+2|\nabla\Ric|^2\nonumber\\
&=&4\lambda|\Ric|^2-4\RRR_{ijks}\RRR_{ik}\RRR_{js}+2|\nabla \Ric|^2,\nonumber
\end{eqnarray}
where in the last inequality we have used the fact that
\[
\ \RRR_{is}\RRR_{ik}(X_{sk}-X_{ks})=\RRR_{sk}\RRR_{ik}(X_{si}-X_{is})=0.
\]
Indeed, both quantities are products of a term  symmetric in $s$, $k$ with another antisymmetric in $s$, $k$. Moreover, using \eqref{eq8},
\begin{eqnarray}
\label{XlapS2}\Delta_X\SSS^2&=&2|\nabla \SSS|^2 +2S\Delta_X \SSS\\
&=&2|\nabla \SSS|^2+4\lambda \SSS^2-4\SSS|\Ric|^2.\nonumber
\end{eqnarray}
Hence, by \eqref{T1}, \eqref{T2} \eqref{XlapnormRic} and \eqref{XlapS2} we compute
\begin{eqnarray*}
\frac{1}{2}\Delta_X|\TTT|^2&=&\frac{1}{2}\Delta_X|\Ric|^2-\frac{1}{2m}\Delta_X S^2\\
&=&2\lambda |\Ric|^2-2\RRR_{ijks}\RRR_{ik}\RRR_{js}+|\nabla \Ric|^2-\frac{1}{m}|\nabla \SSS|^2-\frac{2}{m}\lambda \SSS^2+\frac{2}{m} \SSS|\Ric|^2\\
&=&2\lambda |\TTT|^2-2\RRR_{ijks}\RRR_{ik}\RRR_{js}+\frac{2}{m}\SSS|\Ric|^2+|\nabla \TTT|^2.
\end{eqnarray*}
Observe now that, using the splitting of the curvature tensor \eqref{27},
\begin{eqnarray}
\RRR_{ijks}\RRR_{ik}\RRR_{js} &=& \WWW_{ijks}\RRR_{ik}\RRR_{js}+\frac{(2m-1)}{(m-1)(m-2)}\SSS|\Ric|^2\label{Comp1}\\
&&-\frac{2}{m-2}\mathrm{Tr}(\Ric^3)-\frac{\SSS^3}{(m-1)(m-2)}.\nonumber
\end{eqnarray}
Moreover, since all the traces of the Weyl tensor vanish we have that
\begin{equation}\label{Comp2}
\WWW_{ijks}\RRR_{ik}\RRR_{js}=\WWW_{ijks}\TTT_{ik}\TTT_{js}.
\end{equation}
Using \eqref{Comp1} and \eqref{Comp2} in the above computation we obtain \eqref{XlapT1}.
\end{proof}

\section{Proof of the results}
To prove our main results we first need to introduce some auxiliary analytical lemmas. We adapt the proof of Theorem 1.9 of \cite{prsMemoirs} to show the validity of the next
\begin{lemma}\label{FullOmoriYau}
Assume on the Riemannian manifold $(M, \left\langle \,,\,\right\rangle)$ the existence of a nonnegative $C^2$ function $\gamma$ satisfying the following conditions

\begin{eqnarray}
\label{hp2.2} &\gamma\left(x\right)\rightarrow +\infty \textrm{\,\,\,as\,\,\,} r(x)\rightarrow\infty,\\
\label{hp2.3} &\exists A>0 \textrm{\,\,\,such that\,\,\,} \left|\nabla \gamma\right|\leq A\gamma^{\frac{1}{2}}\textrm{\,\,\, off a compact set,}\\
\label{hp2.4}&\exists B>0 \textrm{\,\,\,such that\,\,\,} \Delta_{X}\gamma\leq B\gamma^{\frac{1}{2}}G\left(\gamma^{\frac{1}{2}}\right)^{\frac{1}{2}} \textrm{\,\,\,off a compact set,}
\end{eqnarray}
with $X\in \mathfrak{X}(M)$, $G$ as in \eqref{hpG}, $A, B$ positive constants.
Then, given any function $u\in C^{2}\left(M\right)$ with $u^{*}=\sup_{M}u<+\infty$, there exists a sequence $\left\{z_{k}\right\}_{k}\subset M$ such that
\begin{equation}\label{O-YMaxPr}
\begin{array}{llll}
&\left(i\right)\,u\left(z_{k}\right)>u^{*}-\frac{1}{k},&\left(ii\right)\,\left|\nabla u\left(z_{k}\right)\right|<\frac{1}{k},&\left(iii\right)\,\Delta_{X}u\left(z_{k}\right)<\frac{1}{k},
\end{array}
\end{equation}
for each $k\in\mathbb{N}$, i.e. the full Omori-Yau maximum principle for $\Delta_{X}$
holds on $\left(M,\left\langle \,,\,\right\rangle\right)$.
\end{lemma}
\begin{remark}\rm{
  For $X\equiv 0$ we recover Theorem 1.9 of \cite{prsMemoirs}.}
\end{remark}
\begin{proof}
We define the function
\[
\ \psi(t)=\exp{\int_{0}^{t}\frac{ds}{\sqrt{G(s)}}}
\]
and we note that, by the properties of $G$, $\psi$ is well defined, smooth, positive ad it satisfies
\begin{eqnarray}
\label{2.5} &\psi(t)\to+\infty \quad\textrm{as}\quad t\to+\infty,\\
\label{2.6} &0<\frac{\psi^\prime(t)}{\psi(t)}\leq \frac{C}{t^{\frac{1}{2}}\sqrt{G(t^{\frac{1}{2}})}},
\end{eqnarray}
for some constant $C>0$. Next, we fix a point $p\in M$ and, for each $k\in \mathbb{N}$ and $u$ as in the statement of the lemma, we define
\begin{equation}\label{2.7}
f_{k}(x)=\frac{u(x)-u(p)+1}{\psi(\gamma(x))^{\frac{1}{k}}}.
\end{equation}
Then, $f_{k}(p)>0$. From $u^{*}<+\infty$ and $\psi(\gamma(x))\to+\infty$ as $r(x)\to+\infty$ we deduce that $\limsup_{r(x)\to+\infty}f_{k}(x)\leq 0$.
Thus $f_{k}$ attains a positive absolute maximum at some $x_k\in M$. In this way we produce a sequence $\left\{x_{k}\right\}$. Proceeding as in the proof of Theorem 1.9 of \cite{prsMemoirs}, (up to passing to a subsequence if needed)
\begin{equation}\label{2.8}
u(x_{k})\to u^{*}\quad\textrm{as}\quad k\to+\infty.
\end{equation}
If $\left\{x_{k}\right\}$ remains in a compact set, then $x_{k}\to\bar{x}\in M$ as $k\to+\infty$ (up to passing to a subsequence) and $u$ attains its absolute maximum.\\
At $\bar{x}$ we have
\begin{equation}
\begin{array}{llll}
&u\left(\bar{x}\right)=u^{*};&\left|\nabla u\left(\bar{x}\right)\right|=0;&\Delta_{X}u\left(\bar{x}\right)=\Delta u(\bar{x})-\left\langle X,\nabla u\right\rangle(\bar{x})\leq 0.
\end{array}
\end{equation}
In this case we let $z_{k}=\bar{x}$ for any $k\in\mathbb{N}$. We now consider the case $x_{k}\to\infty$ so that $\gamma(x_{k})\to+\infty$. Since $f_{k}$ attains a positive maximum at $x_{k}$ we have
\begin{equation}\label{2.9}
\begin{array}{lll}
&(i)\,\nabla \log f_k(x_{k})=0;&(ii)\,\Delta_{X}\log f_{k}(x_{k})\leq 0.
\end{array}
\end{equation}
From \eqref{2.9} (i), since
\[
\ \nabla(\log f_{k})=\nabla \left(\log\left(\frac{u(x)-u(p)+1}{\psi(\gamma(x))^{\frac{1}{k}}}\right)\right)=\frac{\nabla u(x)}{u(x)-u(p)+1}-\frac{1}{k}\frac{\psi^{\prime}(\gamma(x))}{\psi(\gamma(x))}\nabla \gamma (x),
\]
we have
\[
\ 0=\nabla(\log f_{k})(x_{k})=\frac{\nabla u (x_{k})}{u(x)-u(p)+1}-\frac{1}{k}\frac{\psi^{\prime}(\gamma(x_{k}))}{\psi(\gamma(x_{k}))}\nabla \gamma(x_{k}),
\]
that is,
\begin{equation}\label{2.9a}
\nabla u(x_{k})=\frac{1}{k}\frac{\psi^{\prime}(\gamma(x_{k}))}{\psi(\gamma(x_{k}))}\nabla \gamma (x_{k})(u(x_{k})-u(p)+1).
\end{equation}
Reasoning as in \cite{prsMemoirs}, page 9, we deduce
\begin{equation}\label{2.9b}
\Delta u(x_{k})\leq \frac{u(x_{k}-u(p)+1)}{k}\left\{\frac{\psi^{\prime}(\gamma(x_{k}))}{\psi(\gamma(x_{k}))}\Delta \gamma (x_{k})+\frac{1}{k}\left[\frac{\psi^{\prime}(\gamma(x_{k}))}{\psi(\gamma(x_{k}))}\right]^2|\nabla \gamma(x_{k})|^2\right\}
\end{equation}
A computation using \eqref{2.9} (i), \eqref{hp2.3} and \eqref{2.6} shows that
\begin{equation}\label{2.10}
|\nabla u|(x_{k})\leq \frac{a}{k}\frac{u(x_{k})-u(p)+1}{\sqrt{G(\gamma^{\frac{1}{2}})}}
\end{equation}
Therefore, using \eqref{2.9a}, \eqref{2.9b}, \eqref{2.6} and the assumptions \eqref{hp2.3} and \eqref{hp2.4} we obtain
\begin{align}\label{2.11}
\Delta_X u(x_{k})=&\Delta u(x_{k})-\left\langle X,\nabla u\right\rangle(x_{k})\\
\leq&\frac{u(x_{k})-u(p)+1}{k}\left\{\frac{\psi^{\prime}(\gamma(x_{k}))}{\psi(\gamma(x_{k}))}\Delta\gamma(x_{k})+\frac{1}{k}\left[\frac{\psi^{\prime}(\gamma(x_{k}))}{\psi(\gamma(x_{k}))}\right]^2|\nabla \gamma(x_{k})|^2\right\}\nonumber\\
&-\frac{1}{k}(u(x_{k})-u(p)+1)\frac{\psi^{\prime}(\gamma(x_{k}))}{\psi(\gamma(x_{k}))}\left\langle \nabla \gamma(x_{k}),X(x_{k})\right\rangle\nonumber\\
=&\frac{u(x_{k})-u(p)+1}{k}\left[\frac{\psi^{\prime}(\gamma(x_{k}))}{\psi(\gamma(x_{k}))}\Delta_X\gamma(x_{k})+\frac{1}{k}\left[\frac{\psi^{\prime}(\gamma(x_{k}))}{\psi(\gamma(x_{k}))}\right]^2|\nabla \gamma|^2(x_{k})\right]\nonumber\\
\leq&\frac{u(x_{k})-u(p)+1}{k}\left[C(\gamma(x_{k})G(\gamma(x_{k})^\frac{1}{2}))^{-\frac{1}{2}}B\gamma(x_{k})^{\frac{1}{2}}G(\gamma(x_{k})^{\frac{1}{2}})^{\frac{1}{2}}\right.\nonumber\\
&\left.+\frac{1}{k}C^2(\gamma(x_{k})G(\gamma(x_{k})^{\frac{1}{2}}))^{-1}A^2\gamma(x_{k})\right]\nonumber\\
=&\frac{u(x_{k})-u(p)+1}{k}\left[CB+\frac{1}{k}C^2A^2G(\gamma(x_{k})^\frac{1}{2})^{-1}\right].\nonumber
\end{align}
and the RHS tends to zero as $k\to+\infty$.
In this case we choose $z_{k}=x_{k}$ and \eqref{2.8}, \eqref{2.10}, \eqref{2.11} together with $u^{*}<+\infty$ show the validity of the lemma.
\end{proof}
The next result follows from the easily verified formula
\begin{equation}\label{2.12}
\Delta_X\varphi(u)=\varphi^\prime(u)\Delta_X u
+\varphi^{\prime\prime}(u)|\nabla u|^2
\end{equation}
valid for $u\in C^2(M)$, $\varphi\in C^2 (\mathbb{R})$ and Lemma \ref{FullOmoriYau} with the same reasoning as in the proof of Theorem 1.31 in \cite{prsMemoirs}; see also Theorem 2.6 in \cite{RRS}. Precisely we have
\begin{lemma}\label{Motomiya}
In the assumptions of Lemma \ref{FullOmoriYau} let $u\in C^2(M)$ be a solution of the differential inequality
\begin{equation}\label{Poisson}
\Delta_X u\geq \rho(u, |\nabla u|)
\end{equation}
with $\rho(t,y)$ continuous in $t$, $C^2$ in $y$ and such that
\[
\ \frac{\partial^2 \rho}{\partial y^2}(t, y)\geq 0\qquad \forall y\geq 0.
\]
 Let $f(t)=\rho(t,0)$. Then a sufficient condition to guarantee that $u^*=\sup_M u<+\infty$, is the existence of a continuous function $F$, positive on $[a, +\infty)$ for some $a\in\mathbb{R}$ and satisfying
\begin{eqnarray}
&\left\{\int_a^tF(s)ds\right\}^{-\frac{1}{2}}\in L^1(+\infty),\nonumber\\
&\limsup_{t\rightarrow+\infty}\frac{\int_a^tF(s)ds}{tF(t)}<+\infty,\nonumber\\
&\liminf_{t\to+\infty}\frac{f(t)}{F(t)}>0,\nonumber\\
&\liminf_{t\to+\infty}\frac{\left\{\int_a^tF(s)ds\right\}^{\frac{1}{2}}}{F(t)}\left.\frac{\partial \rho}{\partial y}\right|_{(t,0)}>-\infty.\nonumber
\end{eqnarray}
Furthermore in this case
\[
\ f(u^*)\leq 0.
\]
\end{lemma}
In the next result we shall use the following generalized Bochner formula which can be easily deduced from the usual one, see \cite{petwylie2}. For $X\in\mathfrak{X}(M)$ and $u\in C^3(M)$,
\begin{equation}\label{BochnerRS}
\frac{1}{2}\Delta_X|\nabla u|^2=|Hess (u)|^2+\left\langle \nabla\Delta_X u, \nabla u\right\rangle+ \Ric_X(\nabla u, \nabla u)
\end{equation}
with
\begin{equation}\label{RicX}
\Ric_X=\Ric+\frac{1}{2}\mathcal{L}_X\left\langle \,,\,\right\rangle.
\end{equation}
Note that, when $X=\nabla f$, $f\in C^\infty(M)$ then $\Ric_{\nabla f}=\Ric_f$, the usual Bakry-Emery Ricci tensor
\[
\ \Ric_f=\Ric+Hess(f).
\]

We are now ready to state a comparison result for the operator $\Delta_X$ which mildly extends Theorem 2.1 (ii) in \cite{Q1} by Z. Qian. 
\begin{lemma}\label{2.18}
Let $(M, \left\langle \,,\,\right\rangle)$ be a complete manifold of dimension $m$ and $X\in\mathfrak{X}(M)$ a vector field satisfying the growth condition
\begin{equation}\label{BoundX2}
|X|\leq \sqrt{G(r)}
\end{equation}
for some positive nondecreasing function $G\in C^1([0,+\infty))$. Suppose that
\begin{equation}\label{GrowthRicX}
\Ric_X(\nabla r,\nabla r)\geq -(m-1)\overline{G}(r),
\end{equation}
for some $C^1$ positive function $\overline{G}$ on $[0,+\infty)$ such that
\begin{equation}\label{hpG32}
\inf_{[0,+\infty)}\frac{\overline{G}^\prime}{\overline{G}^{\frac{3}{2}}}>-\infty.
\end{equation}
Then there exists $A=A(m)>0$ sufficiently large such that
\begin{equation}\label{EstXLaplr}
\Delta_Xr\leq A\sqrt{\overline{G}}+\sqrt{G}
\end{equation}
pointwise in $M\setminus(\left\{o\right\}\cup cut(o))$ and weakly on all of $M$.
\end{lemma}

\begin{proof}
Let $h$ be the solution on $\mathbb{R}^{+}_{0}$ of the Cauchy problem
\begin{equation}\label{CauchyProblemh}
\left\{
\begin{array}{l}
h^{\prime\prime}-\overline{G}h=0\\
h\left(0\right)=0, \,\,\,h^{\prime}\left(0\right)= 1.\\
\end{array}
\right.
\end{equation}
Note that $h>0$ on $\mathbb{R}^{+}$ since $\overline{G}\geq0$. Fix $x\in M\setminus\left(cut\left(o\right)\cup\left\{o\right\}\right)$, with $cut(o)$ the cut locus of $o$, and let $\gamma:\left[0,l\right]\rightarrow M$, $l=length\left(\gamma\right)$, be a minimizing geodesic with $\gamma\left(0\right)=o$, $\gamma\left(l\right)=x$. Note that $\overline{G}\left(r\circ\gamma\right)\left(t\right)=\overline{G}\left(t\right)$. From Bochner formula applied to the distance function $r$ (outside $cut(o)\cup \left\{o\right\}$) we have
\begin{equation}\label{Bochnerr}
\ 0=\left|Hess\left(r\right)\right|^{2}+\left\langle \nabla r,\nabla\Delta r\right\rangle+ Ric\left(\nabla r, \nabla r\right)
\end{equation}
so that, using the inequality
\[
\ |Hess(r)|^2\geq\frac{(\Delta r)^2}{m-1},
\]
 it follows that  the function $\varphi\left(t\right)=\left(\Delta r\right)\circ \gamma\left(t\right)$, $t\in\left(0,l\right]$, satisfies the Riccati differential inequality
\begin{equation}\label{Riccati}
\varphi^{\prime}+\frac{1}{m-1}\varphi^{2}\leq -Ric\left(\nabla r\circ\gamma,\nabla r\circ\gamma\right)
\end{equation}
on $\left(0,l\right]$. With $h$ as in \eqref{CauchyProblemh} and using the definition of $\Ric_{X}$, \eqref{GrowthRicX} and \eqref{Riccati} we compute
\begin{align*}
\left(h^{2}\varphi\right)^{\prime}=&2hh^{\prime}\varphi+h^{2}\varphi^{\prime}\\
\leq&2hh^{\prime}\varphi-\frac{h^{2}\varphi^{2}}{m-1}+h^{2}\left(m-1\right)\overline{G}\left(t\right)+ h^2\left\langle \nabla_{\nabla r}X,\nabla r\right\rangle\circ\gamma\\
=&-\left(\frac{h\varphi}{\sqrt{m-1}}-\sqrt{m-1}h^{\prime}\right)^2+(m-1)(h^{\prime})^2+h^2(m-1)\overline{G}(t)\\
&+h^2(\left\langle X, \dot{\gamma}\right\rangle\circ\gamma)^{\prime}.
\end{align*}
We let
\[
\ \varphi_{\overline{G}}\left(t\right)=\left(m-1\right)\frac{h^{\prime}}{h}\left(t\right)
\]
so that, using \eqref{CauchyProblemh}
\[
\ \left(h^{2}\varphi_{\overline{G}}\right)^{\prime}=\left(m-1\right)\left(h^{\prime}\right)^{2}+\left(m-1\right)\overline{G}\left(t\right)h^{2}.
\]
Inserting this latter into the above inequality we obtain
\begin{equation*}
\left(h^{2}\varphi\right)^{\prime}\leq\left(h^{2}\varphi_{\overline{G}}\right)^{\prime}+h^{2}(\left\langle X, \dot{\gamma}\right\rangle\circ\gamma)^{\prime}
\end{equation*}
Integrating on $\left[0,r\right]$ and using \eqref{CauchyProblemh} yields
\begin{equation}\label{comp1}
h^{2}\left(r\right)\varphi\left(r\right)\leq h^{2}\left(r\right)\varphi_{\overline{G}}\left(r\right)+\int^{r}_{0}h^{2}(\left\langle X,\dot{\gamma}\right\rangle\circ\gamma)^{\prime}.
\end{equation}
Next, we define that
\begin{equation*}
\varphi_{X}=\left(\Delta_{X}r\right)\circ\gamma
=\left(\Delta r\right)\circ\gamma-\left\langle X,\nabla r\right\rangle\circ\gamma=\varphi-\left\langle X,\nabla r\right\rangle\circ\gamma.
\end{equation*}
Thus, using \eqref{comp1}, \eqref{CauchyProblemh} and integrating by parts we compute
\begin{align*}
h^{2}\varphi_{X}(r)\leq&h^{2}\varphi_{\overline{G}}(r)-h^{2}\left\langle X,\nabla r\right\rangle\circ\gamma(r)+\int^{r}_{0}h^{2}(\left\langle X,\dot{\gamma}\right\rangle\circ\gamma)^{\prime}dt\\
=&h^{2}\varphi_{\overline{G}}-h^{2}\left\langle X,\nabla r\right\rangle\circ\gamma(r)+h^{2}\left(\left\langle X, \dot{\gamma}\right\rangle\circ\gamma\right)\left.\right|^{r}_{0}-\int^{r}_{0}\left(h^{2}\right)^{\prime}(\left\langle X,\dot{\gamma}\right\rangle\circ\gamma)dt\\
=&h^{2}\varphi_{\overline{G}}(r)-\int^{r}_{0}\left(h^{2}\right)^{\prime}\left(\left\langle X,\dot{\gamma}\right\rangle\circ\gamma\right)dt,
\end{align*}
that is,
\begin{equation}\label{comp2}
h^{2}\varphi_{X}(r)\leq h^{2}\varphi_{\overline{G}}(r)-\int^{r}_{0}\left(h^{2}\right)^{\prime}(\left\langle X,\dot{\gamma}\right\rangle\circ\gamma)dt
\end{equation}
on $\left(0,l\right]$. Observe now that, by the Cauchy--Schwarz inequality and \eqref{BoundX2},
\begin{equation*}
-\left\langle X,\dot{\gamma}\right\rangle\leq|X|\leq\sqrt{G(r)}.
\end{equation*}
Inserting this into \eqref{comp2}, using the fact that $(h^2)^\prime\geq 0$, integrating by parts and recalling that $G^\prime\geq 0$, we obtain
\[
\ h^{2}\varphi_{X}(r)\leq h^{2}\varphi_{\overline{G}}(r)+h^{2}\sqrt{G(r)}-\int_{0}^r\frac{h^2G^{\prime}(t)}{2\sqrt{G(t)}}dt\leq h^{2}\varphi_{\overline{G}}(r)+h^{2}\sqrt{G(r)}
\]
on $\left(0,l\right]$. It follows that
\begin{equation*}
%\label{3.12}
\varphi_{X}(r)\leq\varphi_{\overline{G}}(r)+\sqrt{G\left(r\right)}
\end{equation*}
on $\left(0,l\right]$. In particular
\begin{equation}\label{comp3}
\Delta_{X}r\left(x\right)\leq\left(m-1\right)\frac{h^{\prime}\left(r\left(x\right)\right)}{h\left(r\left(x\right)\right)}+ \sqrt{G(r(x))}
\end{equation}
pointwise on $M\setminus\left(\left\{0\right\}\cup cut\left(o\right)\right)$.
Proceeding as in Theorem 2.4 of \cite{prsbook} one shows that (\ref{comp3})
holds weakly on all of $M$. Moreover, fixing $\overline{A}\in\mathbb{R}_+$ and defining
\begin{equation*}
\overline{h}(t)=\frac{\overline{A}^{-1}}{\sqrt{G(0)}}\left\{e^{\overline{A}\int_{0}^{t}\sqrt{\overline{G}(s)}ds}-1\right\},
\end{equation*}
we have $\overline{h}(0)=0$, $\overline{h}^{\prime}(0)=1$ and
\begin{equation*}
\overline{h}^{\prime\prime}-\overline{G}\,\overline{h}\geq\frac{\overline{G}}{\sqrt{\overline{G}(0)}}\left[e^{\overline{A}\int_{0}^{t}\sqrt{\overline{G}(s)}ds}\left(\inf_{[0,+\infty)}\frac{\overline{G}^{\prime}}{2\overline{G}}^{\frac{3}{2}}(t)+\overline{A}-\frac{1}{\overline{A}}\right)\right]\geq0
\end{equation*}
by \eqref{hpG32}, for $\overline{A}$ sufficiently large. So by Sturm comparison (see e.g. Lemma 2.1 in \cite{prsbook}),
\begin{equation*}
\frac{h^{\prime}}{h}\leq \frac{\overline{h}^{\prime}}{\overline{h}}=D\sqrt{\overline{G}}\frac{\exp{\left[D\int_0^t\sqrt{\overline{G}}ds\right]}}{\exp{\left[D\int_0^t\sqrt{\overline{G}}ds\right]}-1}\leq D\sqrt{\overline{G}}
\end{equation*}
for some constant $D>0$ sufficiently large. Thus from \eqref{comp3} we finally obtain
\begin{eqnarray*}
\Delta_X r&\leq& (m-1)D\sqrt{\overline{G}}+\sqrt{G}\\
&\leq&A\sqrt{\overline{G}}+\sqrt{G}
\end{eqnarray*}
for some constant $A=A(m)>0$, pointwise in $M\setminus(\left\{o\right\}\cup cut(o))$ and weakly on all of $M$.
\end{proof}

\begin{remark}
\rm{
In case $G$ is non--increasing, inspection of the above proof yields the estimate
\[
\ \Delta_X r\leq A\sqrt{\overline{G}}+\sqrt{G(0)}
\]
instead of \eqref{EstXLaplr}.
}
\end{remark}
\begin{remark}\label{rm15}
\rm{
  Since with the above notation $\varphi=\varphi_X +\left\langle X,\nabla r\right\rangle\circ\gamma$, from \eqref{EstXLaplr} and \eqref{BoundX2} we obtain $\Delta r \leq B\pa{\sqrt{\overline{G}}+\sqrt{G}}$ for some $B\geq\max\left\{A, 2\right\}$. This observation shows that if $\pa{M, \metric, X}$ is a soliton structure on the complete manifold $\pa{M, \metric}$ satisfying \eqref{BoundX} with $G$ as in \eqref{hpG}, applying Theorem 1.9 in \cite{prsMemoirs} with $\gamma(x)=r(x)^2$, we have the validity of the full Omori-Yau maximum principle (and in particular, $\pa{M, \metric}$ is stochastically complete; see for instance the consequence given in Corollary \ref{19} below). In particular, using again Remark \ref{RmkSharp}, one gets that on a gradient Ricci soliton we have always the validity of the full Omori--Yau maximum principle for the operator $\Delta$. This last fact was, very recently, also established in \cite{BPS}. We are grateful to S. Pigola for pointing out this to us. See also \cite{FLGR-OY} for the shrinking gradient case.
}
\end{remark}

\begin{proof}(of Theorem \ref{ThA}). First of all we observe that since $(M, \left\langle \,,\,\right\rangle, X)$ is a soliton structure on $M$,
\[
\ \Ric_X(\nabla r, \nabla r)=\lambda.
\]
Since by assumption $X$ satisfies \eqref{BoundX}, Lemma \ref{2.18} holds. Hence the function $\gamma(x)=r(x)^2\in C^2(M\setminus\mathrm{cut}(o))$, which meets assumptions \eqref{hp2.2}, \eqref{hp2.3} of Lemma \ref{FullOmoriYau}, satisfies also \eqref{hp2.4} out of $\mathrm{cut}(o)$. An ispection of the proof of Lemma \ref{FullOmoriYau} shows that we need $\gamma$ to be $C^2$ only in a neighborhood of the points of the sequence $\left\{x_k\right\}$. Moreover if some $x_k\in \mathrm{cut}(o)$ we can again guarantee the validity of \eqref{O-YMaxPr} along $\left\{x_k\right\}$ by means of an adaptation of the Calabi trick.  We therefore can apply Lemma \ref{Motomiya}. On the other hand, by \eqref{eq8} of Lemma \ref{XlapRicS}, we have
\begin{equation}\label{2.30}
\frac{1}{2}\Delta_X\SSS=\lambda \SSS-|\Ric|^2=\lambda \SSS-\frac{\SSS^2}{m}-\left|\Ric-\frac{\SSS}{m}\left\langle \,,\,\right\rangle\right|^2
\end{equation}
from which, setting $u=-\SSS$ we immediately deduce the differential inequality
\begin{equation}\label{2.31}
\frac{1}{2}\Delta_Xu\geq\lambda u+\frac{u^2}{m}.
\end{equation}
We apply Lemma \ref{Motomiya} to \eqref{2.31} with the choices $F(t)=t^2$,
\[
\ \rho(u,|\nabla u|)=\lambda u+\frac{u^2}{m}.
\]
Then $u^*<+\infty$ and
\begin{equation}\label{2.32}
\lambda u^*+\frac{(u^*)^2}{m}\leq 0.
\end{equation}
But $u^*=-\SSS_*$ so that the claimed bounds on $\SSS_*$ in the statement of Theorem \ref{ThA} follow immediately from \eqref{2.32}.

\textit{Case (i).} Suppose now $\lambda<0$ and that for some $x_0\in M$
\[
\ S(x_0)=S_*=m\lambda.
\]
In particular $S(x)\geq m\lambda$ on $M$ and the function $w=S-m\lambda$ is non--negative. From \eqref{2.30} we immediately see that
\begin{equation}\label{2.33}
\Delta_X w+2\lambda w\leq \Delta_X w +\frac{2S}{m}w\leq 0.
\end{equation}
We let
\[
\ \Omega_0=\left\{x\in M : w(x)=0\right\}.
\]
$\Omega_0$ is closed and non--empty since $x_0\in\Omega_0$. Let $y\in\Omega_0$. By the maximum principle applied to \eqref{2.33}, $w\equiv 0$ in a neighborhood of $y$ so that $\Omega_0$ is open. Thus $\Omega_0=M$ and $S(x)=m\lambda $ on $M$. From equation \eqref{2.30} we then deduce $|\Ric-\frac{S}{m}\left\langle \,,\,\right\rangle|\equiv 0$, that is, $(M, \left\langle \,,\,\right\rangle)$ is Einstein and from \eqref{RicSolEq} $X$ is a Killing field. Analogously, if $S(x_0)=S_*=0$ for some $x_0\in M$ we deduce that $(M, \left\langle \,,\,\right\rangle)$ is Ricci flat and $X$ is a homothetic vector field.\\
\textit{Case (ii).} Suppose $\lambda=0$ and that, for some $x_0\in M$,
\[
\ S(x_0)=S_*=0.
\]
From \eqref{2.30}
\[
\ \Delta_XS\leq -2\frac{S^2}{m}\leq 0.
\]
Since $S(x)\geq S_*=0$ by the maximum principle we conclude $S(x)\equiv 0$. By \eqref{2.30}, $(M, \left\langle \,,\,\right\rangle)$ is Ricci flat and from \eqref{RicSolEq} $X$ is a Killing field.\\
\textit{Case (iii).} Finally, suppose $\lambda>0$. Then $S(x)\geq S_{*}\geq 0$. From \eqref{2.30}
\[
\ \Delta_XS-2\lambda S\leq 0.
\]
If $S(x_0)=S_{*}=0$ for some $x_0\in M$, then again by the maximum
principle $S(x)\equiv 0$. From \eqref{2.30}, $(M, \left\langle
\,,\,\right\rangle)$ is Ricci flat and from \eqref{RicSolEq},
$\mathcal{L}_X\left\langle \,,\,\right\rangle=2\lambda\left\langle
\,,\,\right\rangle$, so that $X$ is a homothetic vector field.
Suppose now $\SSS_*=m\lambda>0$. From \eqref{2.30}
\[
\ \Delta_X S\leq\frac{2}{m} S(m\lambda  -S)
\]
and since $S(x)\geq S_*=m\lambda>0$
\[
\ \Delta_X S\leq 0\quad\textrm{on}\quad M.
\]
By the maximum principle $S(x)\equiv m\lambda$. From \eqref{2.30},
$(M,\left\langle \,,\,\right\rangle)$ is Einstein and
\eqref{RicSolEq} implies that $X$ is a Killing field. Furthermore,
since $\lambda>0$, $(M, \left\langle \,,\,\right\rangle)$ is
compact by Myers' theorem.
\end{proof}
Before proceeding with the proof of Theorem \ref{ThB} we recall the next estimate due to Huisken, see \cite{huisken}, Lemma 3.4:
\begin{equation}\label{huisEst}
  \abs{\TTT_{ik}\TTT_{jk}\WWW_{ijkl}} \leq \frac{\sqrt{2}}{2}\sqrt{\frac{m-2}{m-1}}\abs{W}\abs{T}^2.
\end{equation}
\begin{proof}(of Theorem \ref{ThB}). From \eqref{XlapT1}, Okumura's lemma \cite{okumura} and Huisken's estimate \eqref{huisEst} we obtain
\begin{equation*}
\frac{1}{2}\Delta_X|T|^2\geq |\nabla T|^2+2\pa{\lambda-\frac{m-2}{m(m-1)}\SSS-\sqrt{\frac{m-2}{2(m-1)}}\abs{W}}|T|^2-\frac{4}{\sqrt{m(m-1)}}|T|^3.
\end{equation*}
We set $u=|T|^2$ and from the above we deduce
\begin{equation}\label{2.34}
\frac{1}{2}\Delta_X u\geq 2\pa{\lambda-\frac{m-2}{m(m-1)}\SSS^*-\sqrt{\frac{m-2}{2(m-1)}}\abs{W}^*}u-\frac{4}{\sqrt{m(m-1)}}u^{\frac{3}{2}}.
\end{equation}
We note that if $|T|^*=+\infty$ then \eqref{1.10} is obviously satisfied otherwise $u^*<+\infty$ and in the assumptions of Theorem \ref{ThB} we have the validity of Lemma \ref{FullOmoriYau}. It follows that
\[
\ u^*\left[\frac{1}{2}\left(\lambda\sqrt{m(m-1)}-\frac{m-2}{\sqrt{m(m-1)}}\,\SSS^*-\sqrt{\frac{m(m-2)}{2}}\abs{W}^*\right)-\sqrt{u^*}\right]\leq 0,
\]
from which we deduce that either $u^*=0$, that is, $T\equiv 0$ on $M$, or $|T|^*$ satisfies \eqref{1.10}. In the first case $(M, \left\langle \,,\,\right\rangle)$ is Einstein.
\end{proof}
\begin{remark}
\rm{
As observed (in the gradient case) in \cite{C}, if $\pa{M, \metric}$ is Einstein and in addition it is a shrinking soliton which is not Ricci flat, by Theorem \ref{ThA}, $S$ is a positive constant and thus
  $\pa{M, \metric}$ is compact by Myers' Theorem. In this latter case, if $\abs{\WWW}$ is sufficiently small, precisely if $\abs{\WWW}^2 \leq \frac{4}{(m+1)m(m-1)(m-2)}\SSS^2$, then
  $\pa{M, \metric}$ has positive curvature operator (\cite{huisken}, Corollary 2.5). Thus, from Tachibana \cite{tachibana}, $\pa{M, \metric}$
  has positive constant sectional curvature and it is therefore a finite quotient of $S^m$.}
\end{remark}

In the next Proposition, as a consequence of Lemma \ref{2.18}, we deduce estimates for the volume growth of geodesic spheres and geodesic balls assuming a suitable radial control on $\Ric_X$ and on the growth of the vector field $X$.

\begin{proposition}\label{2.52}
Let $(M, \metric)$ be a complete manifold and $X\in\mathfrak{X}(M)$ a vector field satisfying
\[
\abs{X}\leq \sqrt{G(r)}
\]
for some nonnegative nondecreasing function $G \in C^0\pa{[0, +\infty)}$.
 Assume
\[
\ \Ric_X(\nabla r, \nabla r)\geq -(m-1)\overline{G}(r)
\]
for some smooth positive function $\overline{G}(r)$ on $[0, +\infty)$ such that
\[
\inf_{[0, +\infty)} \frac{\overline{G}'}{\overline{G}^{3/2}} >-\infty.
\]
Then
\begin{equation}\label{vol1}
\ \rm{vol}(\partial B_r)\leq Ce^{B\int_0^r\pa{\sqrt{\overline{G}(s)}+\sqrt{G(s)}} ds}
\end{equation}
for almost every $r$, and, as a consequence,
\begin{equation}\label{vol2}
\ \rm{vol}(B_r)\leq C\int_0^r\pa{e^{B\int_0^t\pa{\sqrt{\overline{G}(s)}+\sqrt{G(s)} ds}}} dt+ D,
\end{equation}
 with $C, B, D$ sufficiently large positive constants.
\end{proposition}
\begin{proof}
We let
\[
\ h(r)=e^{\frac{B}{m-1}\int_0^r\pa{\sqrt{\overline{G}(s)}+\sqrt{G(s)}} ds} -1,
\]
with $B\geq \max \left\{A, 2\right\}$. Since $\Delta r=\Delta_X r+\left\langle X, \nabla r\right\rangle$, from \eqref{EstXLaplr}, \eqref{BoundX} and the choice of $h$ we have
\[
\ \Delta r\leq (m-1)\frac{h^{\prime}(r)}{h(r)}
\]
pointwise on $M\setminus\left(\left\{0\right\}\cup cut\left(o\right)\right)$ and weakly on all of $M$.
 We now follow the argument of the proof of Theorem 2.14 of \cite{prsbook}. Thus for each $\varphi\in Lip_{0}(M)$, $\varphi\geq 0$,
\begin{equation}\label{2.53_68}
-\int\left\langle \nabla r, \nabla \varphi\right\rangle\leq (m-1)\int\frac{h^{\prime}(r(x))}{h(r(x))}\varphi.
\end{equation}
Next, we fix $0<r<R$ and $\varepsilon>0$ and we let $\rho_\varepsilon$ be the piecewise linear function
\begin{equation*}
\rho_\varepsilon(t)=
\begin{cases}0\qquad &\textrm{if }t\in [0,r)\\
\frac{t-r}{\varepsilon}\qquad &\textrm{if }t\in [r,r+\epsilon)\\
1\qquad &\textrm{if }t\in [r+\varepsilon,R-\varepsilon)\\
\frac{R-t}{\varepsilon}\qquad&\textrm{if } t\in[R-\varepsilon,R)\\
0\qquad &\textrm{if }t\in [R,+\infty),
\end{cases}
\end{equation*}
and we define the radial cut-off function
\[
\ \varphi_{\varepsilon}(x)=\rho_\varepsilon(r(x))h(r(x))^{1-m}.
\]
Indicating with $\chi_{s,t}$, $s<t$, the characteristic function of the annulus $B_t\setminus B_s$, we have
\[
\ \nabla\varphi_\varepsilon=\left\{\frac{1}{\varepsilon}\chi_{r, r+\varepsilon}-\frac{1}{\epsilon}\chi_{R-\varepsilon, R}-(m-1)\frac{h^{\prime}(r)}{h(r)}\rho_{\varepsilon}\right\}h(r)^{1-m}\nabla r
\]
for a.e. $x\in M$. Therefore, using $\varphi_\varepsilon$ into \eqref{2.53_68} and simplifying we get
\[
\ \frac{1}{\varepsilon}\int_{B_R\setminus B_{R-\varepsilon}}h(r)^{1-m}\leq \frac{1}{\varepsilon}\int_{B_{r+\varepsilon}\setminus B_r}h(r)^{1-m}.
\]
Using the co--area formula
\[
\ \frac{1}{\varepsilon}\int_{R-\varepsilon}^{R} h(t)^{1-m}\mathrm{vol}(\partial B_t)dt\leq \frac{1}{\varepsilon}\int_r^{r+\varepsilon} h(t)^{1-m}\mathrm{vol}(\partial B_t)dt
\]
and letting $\epsilon\searrow0^+$
\[
\ \frac{\rm{vol}(\partial B_R)}{h(R)^{m-1}}\leq \frac{\rm{vol}(\partial B_r)}{h(r)^{m-1}},
\]
for a.e. $0<r<R$. Letting $r\rightarrow 0$, recalling that
$\rm{vol}(\partial B_r)$ $\sim c_mr^{m-1}$ and $h(r)\sim r$, we
obtain that
\[
\rm{vol}(\partial B_r) \leq C e^{B\int_{0}^R\pa{\sqrt{\overline{G}(s)}+\sqrt{G(s)}}ds}
\]
for some constant $C>0$ and a.e. $R$. Using the co--area formula we then also deduce \eqref{vol2}.
\end{proof}
With the aid of the above estimates we can again conclude that
\begin{corollary}\label{2.54}
In the assumptions of Proposition \ref{2.52} if $G(r)$ and $\overline{G}(r)$ are of the form $r^2\prod_{j=1}^{N}\pa{log^{(j)}(r)}^2$ for $r\gg1$, then $(M, \metric)$ is stochastically complete.
\end{corollary}
\begin{proof}
From Proposition \ref{2.52} and the assumption on $G(r)$ and $\overline{G}(r)$
\[
\ \frac{r}{\log\rm{vol}(B_r)}\notin L^1(+\infty)
\]
and we can apply Grigor'yan sufficient condition, \cite{G}, for stochastic completeness.
\end{proof}

An immediate consequence of Corollary \ref{2.54} is, for instance, the following

\begin{corollary}\label{19}
If $(M, \left\langle \,,\,\right\rangle)$ is a complete manifold supporting a generic Ricci soliton structure satisfying \eqref{BoundX} then it can not be minimally immersed into a non--degenerate cone of the Euclidean space.
\end{corollary}

\begin{proof}
By \eqref{BoundX} and the above Corollary \ref{2.54}, or by Remark
\ref{rm15}, $(M, \left\langle \,,\,\right\rangle)$ is
stochastically complete. Then the result follows from Theorem 1.4
in \cite{MaR}.
\end{proof}

\bibliographystyle{plain}
\bibliography{BiblioNGSolMRRSub}

\end{document}